\documentclass[11pt,reqno]{amsart}
\allowdisplaybreaks

\usepackage[english]{babel}
\usepackage[latin1]{inputenc}
\usepackage[T1]{fontenc}
\usepackage{amssymb,amsmath,amstext,amsfonts,amsthm}
\usepackage{mathtools}
\usepackage{verbatim}
\usepackage{relsize}
\usepackage{graphicx,color}
\usepackage[dvipsnames]{xcolor}
\usepackage[all]{xy}
\usepackage[colorlinks=true, urlcolor=blue, linkcolor=blue, citecolor=blue]{hyperref}
\usepackage{fancyvrb}
\usepackage{tikz-cd}
\usepackage{comment}
\usepackage{hhline}
\usepackage{diagbox}
\usepackage{epigraph}
\usepackage[left=25mm, right=25mm, top=30mm, bottom=30mm, includefoot, includehead]{geometry}

\numberwithin{equation}{section}

\theoremstyle{plain}
\newtheorem{Theorem}{Theorem}
\numberwithin{Theorem}{section}
\newtheorem{Corollary}[Theorem]{Corollary}
\newtheorem{Lemma}[Theorem]{Lemma}
\newtheorem{Proposition}[Theorem]{Proposition}
\newtheorem{Conjecture}[Theorem]{Conjecture}
\theoremstyle{definition}
\newtheorem{Definition}[Theorem]{Definition}
\newtheorem{Example}[Theorem]{Example}
\newtheorem*{Notation}{Notation}

\theoremstyle{remark}

\newcommand{\C}{\mathbb{C}}
\newcommand{\NN}{\mathbb{N}}

\newcommand{\PP}{\mathbb{P}}
\newcommand{\F}{\mathbb{F}}
\newcommand{\R}{\mathbb{R}}
\newcommand{\mR}{\mathsmaller{\mathbb{R}}}

\newcommand{\xx}{\textbf{x}}
\newcommand{\yy}{\textbf{y}}

\newcommand{\nn}{\textbf{n}}

\newcommand{\sE}{\mathcal{E}}

\newcommand{\sI}{\mathcal{I}}

\newcommand{\sF}{\mathcal{F}}

\newcommand{\sO}{\mathcal{O}}
\newcommand{\red}{\color{red}}
\newcommand{\blue}{\color{blue}}

\newcommand{\ser}{\mathcal E^{(r)}}
\makeatletter
\@namedef{subjclassname@2020}{%
	\textup{2020} Mathematics Subject Classification}
\makeatother

\author{Luca Sodomaco}
\address{Department of Mathematics and Systems Analysis, Aalto University, Espoo, Finland}
\email{lucasodomaco@gmail.com}
\author{Ettore Teixeira Turatti}
\address{Universit\`{a} di Firenze, Dipartimento di Matematica e Informatica, Viale Morgagni 67/A, 50134 Firenze, Italy}
\email{ettore.teixeiraturatti@unifi.it}


\title[The span of singular tuples of a tensor beyond the boundary format]{The span of singular tuples of a tensor \\ beyond the boundary format}

\begin{document}

\maketitle

\begin{abstract}
A singular $k$-tuple of a tensor $T$ of format $(n_1,\ldots,n_k)$ is essentially a complex critical point of the distance function from $T$ constrained to the cone of tensors of format $(n_1,\ldots,n_k)$ of rank at most one. A generic tensor has finitely many complex singular $k$-tuples, and their number depends only on the tensor format. Furthermore, if we fix the first $k-1$ dimensions $n_i$, then the number of singular $k$-tuples of a generic tensor becomes a monotone non-decreasing function in one integer variable $n_k$, that stabilizes when $(n_1,\ldots,n_k)$ reaches a boundary format.

In this paper, we study the linear span of singular $k$-tuples of a generic tensor. Its dimension also depends only on the tensor format. In particular, we concentrate on special order three tensors and order-$k$ tensors of format $(2,\ldots,2,n)$. As a consequence, if again we fix the first $k-1$ dimensions $n_i$ and let $n_k$ increase, we show that in these special formats, the dimension of the linear span stabilizes as well, but at some concise non-sub-boundary format. We conjecture that this phenomenon holds for an arbitrary format with  $k>3$.
Finally, we provide equations for the linear span of singular triples of a generic order three tensor $T$ of some special non-sub-boundary format. From these equations, we conclude that $T$ belongs to the linear span of its singular triples, and we conjecture that this is the case for every tensor format.
\end{abstract}

\section{Introduction}\label{sec: intro}

A {\em singular $k$-tuple} of an order-$k$ tensor is the generalization of the notion of singular pair of a rectangular matrix. Singular $k$-tuples preserve important properties of singular pairs. For instance, in the problem of minimizing the distance between a given tensor $T$ and the algebraic variety of rank-one tensors, the singular $k$-tuples of $T$ correspond to the constrained critical points of the distance function. Therefore, they essentially provide an answer to the so-called {\em best rank-one approximation problem} for $T$ \cite{L}.

In Definition \ref{def: Z_T}, we introduce the projective variety $Z_T$ of rank-one tensors corresponding to the singular $k$-tuples of a given tensor $T$.
When $T$ is sufficiently generic, the variety $Z_T$ is zero-dimensional and consists of simple points. What is more, the variety $Z_T$ is {\em degenerate}, namely it is contained in some proper subspace of the ambient tensor space. Therefore, our primary goal is to study the projective span $\langle Z_T\rangle$ of the set $Z_T$.

Recently, the linear space $\langle Z_T\rangle$ has been compared to another important linear space associated with $T$.
In particular, in \cite[Section 5.2]{OP} the authors introduced the {\em singular space} $H_T$ of a tensor $T$. In Definition \ref{def: critical space} we call it the {\em critical space} as in \cite[Definition 2.8]{DOT}.
Due to its relevance in Euclidean distance optimization, more recently Ottaviani \cite{ottaviani2022critical} defined critical spaces of algebraic varieties invariant for the action of a Lie subgroup of the orthogonal group. In few words, the equations of $H_T$ can be obtained from the equations defining singular $k$-tuples without restricting to rank-one solutions.

In \cite[Proposition 2.12]{DOT}, the authors show that the critical space $H_T$ of a generic tensor $T$ contains all the best rank-$k$ approximations of $T$. Furthermore, they show in \cite[Proposition 3.6]{DOT} that, if the ambient tensor space is of {\em sub-boundary format} (see Definition \ref{def: boundary format}), then the projectivization of $H_T$ coincides with the span $\langle Z_T\rangle$ of the singular $k$-tuples of $T$.
In particular, this identification makes the problem of writing equations for $\langle Z_T\rangle$ very easy. As an immediate consequence, the tensor $T$ itself belongs to $\langle Z_T\rangle$. 

The techniques used in \cite{DOT} are nontrivial to extend when the boundary format condition is relaxed. More precisely, the vanishing of the cohomology spaces in \cite[Lemma 3.2]{DOT} does not hold anymore. Nevertheless, the authors observe that still $T\in\langle Z_T\rangle$ in the tensor format $(2,2,4)$, although the subspace $\langle Z_T\rangle$ has codimension one in the projectivized critical space. This suggests that, beyond the boundary format, the singular $k$-tuples satisfy extra linear relations than the ones of $H_T$.

Our motivating problem is finding which are the extra linear relations satisfied by the singular $k$-tuples when the boundary format condition is dropped.
On one hand, we perform the cohomology computations presented in \cite{DOT} for some ``defective'' tensor formats to estimate the dimensional gap between $\langle Z_T\rangle$ and the projectivization of $H_T$. On the other hand, in some examples we provide explicitly the equations of $\langle Z_T\rangle$ which are linearly independent from the ones of $H_T$, and our cohomology computations allow us to guarantee that the new equations are sufficient to obtain $\langle Z_T\rangle$. This part of the paper is closely related to the characterization of determinantal relations among singular $k$-tuples that is studied in the recent paper \cite{beorchia2022equations}.

More precisely, in this paper we present results for the order $\ell+1$ tensor format $(2,\ldots,2,n)$ as well as for the order three tensor format $(2,3,n)$. We are currently working on generalizations of the presented results to any format.
In Lemmas \ref{lemma: coh1}, \ref{lemma: coh2} and \ref{lemma: coh3} we study the vanishing of the cohomology spaces used in \cite{DOT} for the order $\ell+1$ tensor format $(2,\ldots,2,n)$ and we extend their results beyond the boundary format. Our first main result is Theorem \ref{thm: 2...2k+2}, where we show that for $\ell\ge 4$ and $n=\ell+2$ it still holds that $H_T=\langle Z_T\rangle$. With similar techniques we derive Theorem \ref{thm: format 23n} that estimates the dimension of $\langle Z_T\rangle$ in the order three format $(2,3,n)$. Furthermore, in this format and in the format $(2,2,n)$ we study the extra relations satisfied by $\langle Z_T\rangle$ that give its defective dimension when compared with the critical space $H_T$ for the tensor formats $(2,2,n)$ and $(2,3,n)$. This allows us to show in Theorem \ref{thm: membership} that $T\in \langle Z_T\rangle$. We confirm numerically our results with a Julia code for the computation of singular $k$-tuples for any tensor format.

Our results have another interpretation. The second author showed in \cite[Theorem 1.3]{turatti2021tensors} that the generic fiber of the rational map sending a tensor $T$ to the zero dimensional scheme of its singular $k$-tuples consists only of $T$, provided that the boundary format condition is satisfied. In Theorem \ref{thm: fiberbf} we extend this fact to the above-mentioned ``defective'' tensor formats assuming Conjecture \ref{con: Tspan} holds. Our argument relies only on the fact that $T\in \langle Z_T\rangle$. This shows that solving the membership problem $T\in \langle Z_T\rangle$ leads to a full generalization of \cite[Theorem 1.3]{turatti2021tensors}. Corollary \ref{cor: fiber} gives a concrete generalization of \cite[Theorem 1.3]{turatti2021tensors} for some tensor formats.  

\medskip
Our paper is structured as follows. In Section \ref{sec: prelim} we set up our notations and recall the definition of singular tuples of tensors as well as some useful known results. In Section \ref{sec: cohomology prelim} we describe the cohomology tools used to compute the dimension of the span of singular tuples. In Section \ref{sec: def Z_T and H_T} we introduce the span of singular tuples as well as the critical space of a tensor, and we review the state of the art about their relations. Section \ref{sec: cohomology} is the core of our paper, where we compute the dimension of the span of singular tuples beyond the boundary format. In Section \ref{sec: relations} we derive explicit equations for the span of singular tuples in special formats, which allow us to conclude that, in those formats, the tensor $T$ belongs to the span of its singular tuples.
{Finally, in Section \ref{sec: julia} we present code implemented in Julia for the numerical computation of the singular vector tuples of a tensor for any format, we also provide a couple of examples that hopefully will motivate further research on this topic.}

\section{Preliminaries on singular tuples of tensors}\label{sec: prelim}

\begin{Notation}
We often use the shorthand $[k]$ to denote the set of indices $\{1,\ldots,k\}$. Throughout the paper, if not specified we denote by ${\bf j}$ a vector $(j_1,\ldots,j_k)$ of $k$ variables and we set $|{\bf j}|\coloneqq j_1+\cdots+j_k$. Define ${\bf 1} = (1,\ldots,1)\in \NN^k$ and, for $m\in \NN$, let $m{\bf 1} = (m,\ldots, m)\in \NN^k$. 
\end{Notation}

For every $i\in[k]$ we consider an $n_i$-dimensional vector space $V_i$ over the field $\F=\R$ or $\F=\C$. If $\F=\R$, then we prefer the notation $V_i^\mR$.
We denote by $V$ the tensor product $\bigotimes_{i=1}^k V_i$.
This is the space of {\em tensors of format $\nn=(n_1,\ldots,n_k)$}.

\begin{Definition}\label{def: Segre embedding}
A tensor $T\in V$ is of {\it rank-one} (or {\it decomposable}) if $T = \xx_1\otimes\cdots\otimes \xx_k$ for some vectors $\xx_{j}\in V_j$ for all $j\in[k]$. Tensors of rank at most one in $V$ form the affine cone over the {\it Segre variety of format $n_1\times \cdots \times n_k$}, that is the image of the projective morphism
\[
\mathrm{Seg}\colon \PP(V_1)\times \cdots \times \PP(V_k)\to\PP(V)
\]
defined by $\mathrm{Seg}([\xx_1],\ldots,[\xx_k])\coloneqq[\xx_1\otimes\cdots\otimes \xx_k]$ for all non-zero $\xx_j\in V_j$. 
\end{Definition}

Throughout the paper, we adopt the shorthand $\PP=\mathrm{Seg}(\PP(V_1)\times \cdots \times \PP(V_k))$ to indicate the Segre variety introduced before. Furthermore, we abuse notation identifying a tensor $T\in V$ with its class in the projective space $\PP(V)$.

On each projective space $\PP(V_i)$ we fix a smooth projective quadric hypersurface $Q_i=\mathcal{V}(q_i)$, where $q_i$ is the homogeneous polynomial in $\C[x_{i,1},\ldots,x_{i,n_i}]_2$ associated to a positive definite real quadratic form $q_i^\mR\colon V_i^\mR\to\R$. We refer to $Q_i$ as the {\em isotropic quadric} in the $i$-th factor $\PP(V_i)$. We will always assume that $q_i^\mR(\xx_i)=x_{i,1}^2+\cdots+x_{1,n_i}^2$ for all $i\in[k]$.

\begin{Definition}\label{def: Frobenius inner product}
The {\it Frobenius (or Bombieri-Weyl) inner product} of two complex decomposable tensors $T = \xx_1\otimes\cdots\otimes \xx_k$ and $T' = \yy_1\otimes\cdots\otimes \yy_k$ of $V$ is
\begin{equation}\label{eq: Frobenius inner product for tensors}
q_F(T, T')\coloneqq q_1(\xx_1, \yy_1)\cdots q_k(\xx_k, \yy_k)\,,
\end{equation}
and it is naturally extended to every vector in $V$.
We identify all the vector spaces with their duals using the Frobenius inner product.
\end{Definition}

\begin{Definition}\label{def: singular vector tuple}
Let $T\in V$. A {\em singular (vector) $k$-tuple} of $T$ is a $k$-tuple $(\xx_1,\ldots,\xx_k)$ of non-zero vectors $\xx_i\in V_i$ such that
\begin{equation}\label{eq: def singular vector tuple matrix}
\mathrm{rank}
\begin{pmatrix}
T(\xx_1\otimes\cdots\otimes \xx_{i-1}\otimes \xx_{i+1}\otimes\cdots\otimes \xx_k)\\
\xx_i
\end{pmatrix}
\le 1\quad\forall\,i\in[k]\,,
\end{equation}
where
\begin{equation}\label{eq: def contraction}
T(\xx_1\otimes\cdots\otimes \xx_{i-1}\otimes \xx_{i+1}\otimes\cdots\otimes \xx_k)\coloneqq\sum_{j_\ell\in[n_\ell]}t_{j_1\cdots\,j_i\cdots\,j_k}\,x_{1,j_1}\cdots \widehat{x_{i,j_i}}\cdots x_{k,j_k}
\end{equation}
is the tensor contraction of $T=(t_{j_1,\ldots,j_k})$ with respect to $\xx_1\otimes\cdots\otimes \xx_{i-1}\otimes \xx_{i+1}\otimes\cdots\otimes \xx_k$. The symbol $\widehat{x_{i,j_i}}$ in \eqref{eq: def contraction} means that the variable $x_{i,j_i}$ is omitted in the product. 
If we interpret $T$ as a multi-homogeneous polynomial in the coordinates of each vector $\xx_i$, then the previous tensor contraction corresponds to the gradient $\nabla_iT$ with respect to the vector $\xx_i=(x_{i,1},\ldots,x_{i,n_i})$.

A singular $k$-tuple $(\xx_1,\ldots,\xx_k)$ is {\em normalized} if $q_i(\xx_{i})=1$ for all $i\in[k]$. A singular $k$-tuple $(\xx_1,\ldots,\xx_k)$ is {\em isotropic} if $q_i(\xx_{i})=0$ for some $i\in[k]$.
\end{Definition}

The number of singular $k$-tuples of a tensor $T$ of format $\nn$ is constant if the tensor $T$ is generic. It is computed in the following theorem.

\begin{Theorem}{\cite[Theorem 1]{FO}}\label{thm: FO formula}
Let $T\in V$ be a generic tensor. Then $T$ has exactly $\mathrm{ed}(\nn)$ simple singular
tuples, where $\mathrm{ed}(\nn)$ equals the coefficient of the monomial $h_1^{n_1-1}\cdots h_k^{n_k-1}$ in the polynomial
\[
\prod_{i=1}^k\frac{\widehat{h}_i^{n_i}-h_i^{n_i}}{\widehat{h}_i-h_i},\quad\widehat{h}_i\coloneqq\sum_{j\neq i}^k h_j\,.
\]
The number $\mathrm{ed}(\nn)$ coincides with the ED degree of the Segre variety $\PP\subset\PP(V)$ with respect to the Frobenius inner product in $V$.
\end{Theorem}

We refer to \cite{DHOST} for more details on ED degrees of algebraic varieties. 

\begin{Definition}
A tensor $T\in V$ is said to be {\em concise} if there is no proper linear subspace $L_i$ such that $T\in V_1\otimes\dots\otimes L_i\otimes \dots\otimes V_k$ for every $i\in[k]$. The tensor space $V$ is {\em concise} if there exists a tensor $T\in V$ such that $T$ is concise. 
\end{Definition}

If $V$ is non-concise, then for every tensor $T\in V$ there exist linear subspaces $L_i\subset V_i$ such that $T\in L=\bigotimes_{i=1}^k L_i$ and $L$ is a concise tensor space. Moreover, it is a classical result that the space $V$ is concise if and only if $n_i\leq \prod_{j\neq i}n_j$ for every $i\in[k]$.

Theorem \ref{thm: FO formula} tells us that the number $\mathrm{ed}(\nn)$ of singular $k$-tuples of a generic tensor is finite, and its value depends only on the format $\nn$. To study the number $\mathrm{ed}(\nn)$ and later the linear span of singular $k$-tuples, we need to introduce the following tensor format terminology.

\begin{Definition}\label{def: boundary format}
Consider a tensor space $V$ of format $\nn=(n_1,\ldots,n_k)$. Then $\nn$ is
\begin{enumerate}
    \item a {\em sub-boundary format} if for all $i\in[k]$ we have $n_i\le 1+\sum_{j\neq i}(n_j-1)$.
    \item a {\em boundary format} if for some $i\in[k]$ we have $n_i=1+\sum_{j\neq i}(n_j-1)$.
    \item a {\em concise format} if for all $i\in[k]$ we have $n_i\le\prod_{j\neq i}n_j$. Otherwise we say that $\nn$ is a {\em non-concise format}. In particular, if $\nn$ is a non-concise format, then for every tensor $T\in V$ there exists a tensor subspace $V'\subset V$ of concise format $\nn'=(n_1',\ldots,n_k')$ such that $T\in V'$.
\end{enumerate}
\end{Definition}

We recall from \cite[Chapter 1]{GKZ} the notion of dual variety of a projective variety.

\begin{Definition}\label{def: dual varieties}
Let $X\subset \PP(W)$ be a projective variety, where $\dim(W)=n$.
Its {\it dual variety} $X^{\vee}\subset\PP(W^*)$ is the closure of all hyperplanes tangent to $X$ at some smooth point.
The {\it dual defect} of $X$ is the natural number $\delta_X \coloneqq n-2-\dim(X^{\vee})$. A variety $X$ is said to be {\it dual defective} if $\delta_X>0$. Otherwise, it is {\it dual non-defective}. When $X = \PP(W)$, taken with its tautological embedding into itself, $X^{\vee} = \emptyset$ and $\mathrm{codim}(X^{\vee}) = n$. 
\end{Definition}

Of particular interest is the characterization of non-defectiveness of the Segre variety $\PP\subset\PP(V)$ given in the following result.

\begin{Theorem}{\cite[Chapter 14, Theorem 1.3]{GKZ}}
Let $\PP\subset\PP(V)$ be the Segre variety of format $\nn=(n_1,\ldots,n_k)$.
Then $\PP$ is dual non-defective if and only if $V$ is of sub-boundary format.
\end{Theorem}

\begin{Definition}\label{def: hyperdeterminants}
Let $\PP\subset\PP(V)$ be the Segre variety of format $\nn=(n_1,\ldots,n_k)$.
When $\PP$ is dual non-defective, the polynomial equation defining the hypersurface $\PP^\vee\subset\PP(V^*)$ (up to scalar multiples) is called the {\it hyperdeterminant} of format $\nn$ and is denoted by $\mathrm{Det}$. The hyperdeterminant of format $\nn=(n,\ldots,n)$ is said to be {\it hypercubical}. 
\end{Definition}

Suppose now that, given a tensor format $\nn=(n_1,\ldots,n_k)$, the last dimension $n_k$ is sent to infinity. One verifies from Theorem \ref{thm: FO formula} that the value $\mathrm{ed}(\nn)$ stabilizes as long as $n_k$ becomes equal to $n_1+\cdots+n_{k-1}$, namely when $\nn$ becomes a boundary format. This combinatorial phenomenon has a deeper geometric counterpart, which is established by the following theorem.

\begin{Theorem}{\cite[Theorem 4.13]{OSV}}\label{thm: specialization}
Let $N=1+\sum_{i=1}^{k-1}(n_i-1)$ and $m\ge N$. Let $\mathrm{Det}$ be the hyperdeterminant in the boundary format $(n_1,\ldots,n_{k-1},N)$.
Consider a tensor $T\in \bigotimes_{i=1}^{k-1}V_i\otimes\C^{N+1}\subset \bigotimes_{i=1}^{k-1}V_i\otimes\C^{m+1}$ with $\mathrm{Det}(T)\neq 0$. Then the critical points of $T$ on the Segre variety $\prod_{i=1}^{k-1}\PP(V_i)\times\PP(\C^{m+1})$ lie in the subvariety $\prod_{i=1}^{k-1}\PP(V_i)\times\PP(\C^{N+1})$.
\end{Theorem}

We know that the number of singular $k$-tuples of a tensor of format $\nn=(n_1,\ldots,n_k)$ stabilizes for sufficiently large $n_k$, more precisely when $\nn$ becomes a boundary format. This implies that also the linear span of singular $k$-tuples stabilizes for sufficiently large $n_k$. It is natural to ask when exactly this stabilization occur. In particular, in the upcoming sections we will address the following question: {\em does the dimension of the linear span defined by the singular tuples stabilize at a boundary format?} We will see that the answer is negative for many formats. Furthermore, we will observe that the dimension of this linear span will eventually stabilize for some value of $n_k$ such that $\nn$ is neither a boundary format nor the last concise format.

\section{Cohomological preliminaries}\label{sec: cohomology prelim}
We recall some classical concepts and results that we apply throughout the paper. We refer to \cite{Weyman} for more details.

\begin{Notation}
Let $\PP$ be the Segre variety of Definition \ref{def: Segre embedding}. For every $i\in[k]$, we consider the projection $\pi_i\colon\PP\to\PP(V_i)$. Furthermore, we denote by $\mathcal{Q}_i$ the quotient bundle on $\PP(V_i)$. The fiber of $\mathcal{Q}_i$ over $[\xx_i]\in\PP(V_i)$ is $V_i/\langle\xx_i\rangle$. For any vector bundle $\mathcal{B}$ on $\PP$, we use the shorthand $\mathcal{B}({\bf 1})$ to denote the tensorization $\mathcal{B}\otimes\sO({\bf 1})=\mathcal{B}\otimes\sO(1,\ldots,1)$.

In general, if $X\subset\PP(W)$ is any projective variety and $\mathcal{B}$ is a vector bundle on $X$, for all $i\ge 0$ we denote by $H^i(X,\mathcal{B})$ the $i$-th cohomology group of $\mathcal{B}$. If it is clear from the context, we also use the shorthand $H^i(\mathcal{B})$ and we call $h^i(\mathcal{B})\coloneqq\dim(H^i(\mathcal{B}))$.
\end{Notation}

\begin{Theorem}[K\"unneth's formula]\label{thm: Kunneth}
Consider a vector space $V$ of order-$k$ tensors and the Segre variety $\PP\subset\PP(V)$. For all $i\in[k]$ let $\mathcal{B}_i$ be a vector bundle on $\PP(V_i)$. Then for all $q\in \mathbb{Z}_{\ge 0}$
\begin{equation}\label{eq: Kunneth}
H^q\left(\PP,\bigotimes_{i=1}^k \pi_i^\ast \mathcal{B}_i\right)\cong \bigoplus_{|{\bf j}|=q}\bigotimes_{i=1}^k H^{j_i}(\PP(V_i), \mathcal{B}_i)\,.
\end{equation}
\end{Theorem}

Let $G$ be a semisimple simply connected group, let $P\subset G$ be a parabolic subgroup. Let $\Phi^+$ be the set of positive roots of $G$. Let $\delta=\sum \lambda_i$ be the sum of all the fundamental weights and let $\lambda$ be a weight. Let $E_\lambda$ be the homogeneous bundle arising from the irreducible representation of $P$ with highest weight $\lambda$ and $(\cdot,\cdot)$ be the Killing form.
\begin{Definition}
The {\em fundamental Weyl chamber} is the convex set
\begin{equation}\label{eq: Weyl chamber}
\mathcal C=\{\lambda\text{ is a weight} \mid (\lambda,\alpha)\ge0,\  \forall\ \alpha\in \Phi^+ \}.
\end{equation}
\end{Definition}

\begin{Definition}
The weight $\lambda$ is called {\em singular} if there exists a root $\alpha \in \Phi^+$ such that $(\lambda,\alpha)=0$. Otherwise, if $(\lambda,\alpha)\neq 0$ for all the roots $\alpha \in \Phi^+$, we say that $\lambda$ is {\em regular of index $p$} if there exist exactly $p$ roots $\alpha_1,\dots,\alpha_p\in\Phi^+$ such that $(\lambda,\alpha)<0$.
\end{Definition}

\begin{Theorem}[Bott]\label{thm: Bott}
The following are true:
\begin{enumerate}
    \item If $\lambda + \delta$ is singular, then $H^i(G/P,E_\lambda)=0$ for all $i$.
    \item If $\lambda+\delta $ is regular of index $p$, then $H^i(G/P,E_\lambda)=0$ for $i\neq p$. Furthermore $H^p(G/P,E_\lambda)=G_{w(\lambda+\delta)-\delta}$, where $w$ is the unique element of the fundamental Weyl chamber of G which is congruent to $\lambda+\delta$ under the action of the Weyl group.
\end{enumerate}
\end{Theorem}

The next proposition is a direct consequence of Theorem \ref{thm: Bott}.

\begin{Proposition}[Bott's formulas]
Let $\Omega^r_m(d)$ be the $\sO(d)$-twisted sheaf of differential $r$-forms on an $m$-dimensional projective space. For $q,m,r\in \mathbb{Z}_{\ge 0}$ and $d\in \mathbb{Z}$ it holds:
\begin{equation}\label{eq: coh Omega}
    h^q(\Omega_{m}^r(d))=
    \begin{cases}
    \binom{d+m-r}{d}\binom{d-1}{r} & \text{if $q=0\leq r\leq m$ and $d>r$}\\
    1 & \text{if $0\leq q=r\leq m$ and $d=0$};\\
    \binom{-d+r}{-d}\binom{-d-1}{m-r} & \text{if $q=m\ge r\ge 0$ and $d<r-m$}\\
    0 & \text{otherwise.}
    \end{cases}
\end{equation}
\end{Proposition}

\section{The span of singular tuples and the critical space of a tensor}\label{sec: def Z_T and H_T}

Next, we introduce the most important {subset of rank-one tensors} of this paper.

\begin{Definition}\label{def: Z_T}
Consider a tensor $T\in V$. Then
\begin{equation}\label{eq: def Z_T}
Z_T \coloneqq \{[\xx_1\otimes\cdots\otimes\xx_k]\in\PP(V)\mid\text{$(\xx_1,\ldots,\xx_k)$ is a singular $k$-tuple of $T$}\}\,.
\end{equation}
\end{Definition}

For a generic $T\in V$ we have $\dim(Z_T)=0$ and its cardinality $|Z_T|$ equals the ED degree of the Segre variety $\PP$ computed by the Friedland-Ottaviani formula of Theorem \ref{thm: FO formula}.
Throughout the paper we compare the projective span $\langle Z_T\rangle$ with another important tensor subspace which we recall in the next definition.

\begin{Definition}\label{def: critical space}
The {\it critical space} $H_T$ of a tensor $T\in V$ is the linear subspace of $V$ defined by the equations (in the unknowns $z_{i_1\cdots i_k}$ that serve as linear functions on $V$)
\begin{equation}\label{eq: equations critical space}
\sum_{i_\ell\in[n_\ell]}\left(t_{i_1\cdots\,p\,\cdots\,i_k}\,z_{i_1\cdots\,q\,\cdots\,i_k}-t_{i_1\cdots\,q\,\cdots\,i_k}\,z_{i_1\cdots\,p\,\cdots\,i_k}\right)=0\quad\text{where $1\le p<q\le n_\ell$ and $\ell\in[k]$.}
\end{equation}
\end{Definition}

The equations in \eqref{eq: equations critical space} are obtained after computing the two by two minors of the matrix in \eqref{eq: def singular vector tuple matrix} and substituting the relations $z_{j_1\cdots j_k}=x_{1,j_1}\cdots x_{k,j_k}$. In particular, the equations of $H_T$ are linear relations among the elements of $Z_T$, thus in general $\langle Z_T\rangle\subset\PP(H_T)$. Another immediate observation is that $T$ always belongs to $H_T$, as its coordinates always satisfy the equations in \eqref{eq: equations critical space}. We recall a result on the dimension of the critical space. 

\begin{Proposition}{\cite[Proposition 5.6]{OP}}\label{prop: dim H_T}
Consider a tensor $T\in V$ of format $\nn=(n_1,\ldots,n_k)$. Assume $n_1\le\cdots\le n_k$ and let $D=\prod_{i=1}^{k-1}n_i$.
The dimension of the critical space $H_T\subset V$ is
\begin{equation}
\begin{cases}
\prod_{i=1}^kn_i-\sum_{i=1}^k\binom{n_i}{2} & \text{for $n_k\le D$}\\
\binom{D+1}{2}-\sum_{i=1}^{k-1}\binom{n_i}{2} & \text{for $n_k\ge D$}\,.
\end{cases}
\end{equation}
\end{Proposition}

\begin{Proposition}{\cite[Proposition 3.6]{DOT}}\label{prop: equality span critical sub-boundary format}
Consider a generic tensor $T$ in a tensor space $V$ of sub-boundary format. Then $\langle Z_T\rangle=\PP(H_T)$.
\end{Proposition}

Following up Proposition \ref{prop: equality span critical sub-boundary format}, in \cite[Remark 3.7]{DOT} the authors observed that the containment $\langle Z_T\rangle\subset\PP(H_T)$ may become strict beyond the boundary format. Hence they posed the problem of studying the dimension of $\langle Z_T\rangle$ beyond the boundary format. This problem motivated our research. Another problem is to check whether $T\in\langle Z_T\rangle$ even when $\langle Z_T\rangle$ is strictly contained in $\PP(H_T)$.

\medskip
A close friend of $Z_T$ is the following set, which is defined only for generic tensors.

\begin{Definition}\label{def: mathrm{Eig}(T)}
Let $\PP\subset\PP(V)$ be the Segre variety and $T\in V$ be a generic tensor. We define
\begin{equation}
\mathrm{Eig}(T)\coloneqq\{(\xx^{(1)},\ldots,\xx^{(\mathrm{ed}(\nn))})\mid\text{$\xx^{(i)}$ is a singular $k$-tuple of $T$ for all $i\in[k]$}\}
\end{equation}
as a subset of the non-ordered cartesian product $\PP^{\times \mathrm{ed}(\nn)}/S_{\mathrm{ed}(\nn)}$, where $\mathrm{ed}(\nn)$ is the ED degree of $\PP$ computed in Theorem \ref{thm: FO formula} and $S_{\mathrm{ed}(\nn)}$ denotes the symmetric group on $\mathrm{ed}(\nn)$ elements. 
\end{Definition}

\begin{Theorem}{\cite[Theorem 1.3]{turatti2021tensors}}\label{thm: fiber}
Consider a tensor space $V$ of sub-boundary format. Let 
\begin{equation}\label{eq: def tau}
\tau\colon\PP(V)\dashrightarrow\frac{\PP^{\times \mathrm{ed}(\nn)}}{S_{\mathrm{ed}(\nn)}}
\end{equation}
be the rational map sending a tensor $T\in V$ to the locus of singular $k$-tuples $\mathrm{Eig}(T)$. If $T\in V$ is generic, then the fiber $\tau^{-1}(T)$ consists only of $T$.
\end{Theorem}

We generalize the previous result in Theorem \ref{thm: fiberbf}.

\section{Computing the dimension of the span of singular tuples}\label{sec: cohomology}

Consider the notations used in Section \ref{sec: cohomology prelim}. We recall the construction of $Z_T$ as zero locus of a section $\sigma$ of a suitable vector bundle on $\PP$, which is defined as
\begin{equation}
\sE\coloneqq\bigoplus_{i=1}^k\sE_i\,,\quad\sE_i\coloneqq(\pi_i^\ast\mathcal{Q}_i)\otimes\mathcal{O}(1,\ldots,1,\overbracket{0}^i,1,\ldots,1)\quad\forall\,i\in[k]\,.
\end{equation}
We have that $\mathrm{rank}(\sE)=\dim(\PP)=\sum_{i=1}^k(n_i-1)$. For every $i\in[k]$, the tensor $T$ yields a global section of $\sE_i$, which over the point $([\xx_1],\ldots,[\xx_k])\in\PP$ is the map
\[
(\lambda_1\xx_1,\ldots,\lambda_k\xx_k)\in\prod_{i=1}^k\langle\xx_i\rangle \mapsto [T(\lambda_1\xx_1\otimes\cdots\otimes\lambda_{i-1}\xx_{i-1}\otimes\lambda_{i+1}\xx_{i+1}\otimes\cdots\otimes\lambda_k\xx_k)]\in\frac{V_i}{\langle\xx_i\rangle}\,.
\]
Combining these $k$ sections, the tensor $T$ yields a global section $s_T$ of $\sE$. By \cite[Proposition 2.6]{DOT}, if $T$ is generic, then $[\xx_1\otimes\cdots\otimes\xx_k]\in Z_T$ if and only if $([\xx_1],\ldots,[\xx_k])$ is in the zero locus of the section $s_T$.
The section $s_T$ of $\sE$ yields a homomorphism $\sE^\ast\to\sO$ of sheaves whose image is contained in the ideal sheaf $\sI_{Z_T}$ of the zero locus of $s_T$.

\begin{Lemma}{\cite[Lemma 3.2]{DOT}}\label{lemma 3.2 DOT}
Define $\sE^{(r)}\coloneqq\left(\bigwedge^r\sE^\ast\right)\otimes \sO({\bf 1})$ for all integer $r\ge 1$. Then
\begin{equation}\label{eq: iso omega}
\sE^{(r)}\cong\bigoplus_{|{\bf j}|=r}\bigotimes_{i=1}^k\pi_i^\ast \Omega_{n_i-1}^{j_i}(2j_i+1-r)\,.
\end{equation}
\end{Lemma}

Using the isomorphism \eqref{eq: iso omega}, in the following three lemmas we compute the cohomology groups $H^{q}(\ser)$ for the order $\ell+1$ format $\nn=(2,\dots,2,n)$.

\begin{Lemma}\label{lemma: coh1}
Consider a space $V$ of order $\ell+1$ tensors of format $\nn=(2,\ldots,2,n)$ with $n\ge\ell+2$. Given nonnegative integers $r\ge2$ and $q\leq r$, then $H^q(\sE^{(r)})=0$ for all $r<\ell$.
\end{Lemma}
\begin{proof}
We use Theorem \ref{thm: Kunneth} and Lemma \ref{lemma 3.2 DOT} to compute the cohomology of $H^q(\sE^{(r)})$. Notice that
\[
H^0(\Omega^{j_i}_1(1+2j_i-r))=0
\]
for all $0\leq j_i\leq r$. The only possibility for the non-vanishing of $H^q(\sE^{(r)})$ is that
\[
H^1(\Omega^{j_i}_1(1+2j_i-r))\neq 0\,,
\]
that holds for $r>2$. This implies that
\[
H^q(\sE^{(r)})=\left(\bigotimes_{i=1}^\ell H^1(\Omega^{j_i}_1(1+2j_i-r))\right)\otimes H^{j_{\ell+1}}(\Omega^{j_{\ell+1}}_{\ell+1}(1+2j_{\ell+1}-r))\,.
\]
In turn we have $q\ge\ell$, thus the claim holds.
\end{proof}

\begin{Lemma}\label{lemma: coh2}
Consider a space $V$ of order $\ell+1$ tensors of format $\nn=(2,\ldots,2,n)$ with $n\ge\ell+2$.
Given nonnegative integers $r\ge 2$ and $q<r$, we have that
\begin{itemize}
    \item[$(i)$] if $\ell+1\leq r\leq n-1$, then
    \[
    H^q(\sE^{(r)}) =
    \begin{cases}
    0 & \text{if $q\neq \ell$}\\
    \left(\bigotimes_{i=1}^\ell H^1(\Omega^0_1(-r+1))\right)\otimes H^0(\Omega^{r}_n(r+1)) & \text{if $q=\ell$}
    \end{cases}
    \]
    \item[$(ii)$] if $r\ge n-1$, then $H^\ell(\sE^{(r)})=0$.
\end{itemize} 
\end{Lemma}
\begin{proof}
Again we use the fact that $H^i(\Omega^{j_i}_1(2j_i+1-r))\neq 0$ if and only if $i=1$. We consider the cases where the cohomology of $\Omega^{j_{\ell+1}}_{n-1}(1+2j_{\ell+1}-r)$ does not vanish. Notice also that $r-\ell\leq j_{\ell+1}\leq \min\{r,n-1\}$.

\begin{enumerate}
    \item We have $H^0(\Omega^{j_{\ell+1}}_{n-1}(1+2j_{\ell+1}-r))\neq 0$ if and only if $1+2j_{\ell+1}-r>j_{\ell+1}$. Hence $j_{\ell+1}+1>r$, namely $r=j_{\ell+1}$. In such case we have $j_i=0$ for all $i\in[\ell]$.
    \item We have $H^{n-1}(\Omega_{n-1}^{j_{\ell+1}}(1+2j_{\ell+1}-r))\neq 0$ if and only if $1+2j_{\ell+1}-r<j_{\ell+1}-n+1$. Hence $r-\ell+n\leq j_{\ell+1}+n<r$ which yields $n<\ell$, a contradiction.
    \item Finally, we have $H^{j_{\ell+1}}(\Omega^{j_{\ell+1}}_{n-1}(1+2j_{\ell+1}-r))\neq 0$ if and only if $1+2j_{\ell+1}-r=0$. This means that $j_{\ell+1}=\frac{r-1}{2}$. Furthermore $q=\ell+j_{\ell+1}< r$. On the other hand $r=\sum_{i=1}^\ell j_i+j_{\ell+1}\leq \ell+j_{\ell+1}$, so $r\leq \ell+j_{\ell+1}<r$, a contradiction.
\end{enumerate}
Thus the only non-zero cohomology of $\Omega^{j_{\ell+1}}_{n-1}(1+2j_{\ell+1}-r)$ comes from case $(1)$, that corresponds exactly to $H^\ell(\sE^{(r)})=\left(\bigotimes_{i=1}^\ell H^1(\Omega^0_1(-r+1))\right)\otimes H^0(\Omega^{r}_{n-1}(r+1))$. Notice also that if $r\ge n$, then $\Omega^{r}_{n-1}=0$.
\end{proof}

\begin{Lemma}\label{lemma: coh3}
Consider a space $V$ of order $\ell+1$ tensors of format $\nn=(2,\ldots,2,n)$ with $n\ge\ell+2$.
Given nonnegative integers $r\ge2$ and $q\le r$, if $\ell=r$ then the only non vanishing cohomology is
\[
H^\ell(\sE^{(\ell)})=\left(\bigotimes_{i=1}^\ell H^1(\Omega^0_1(-r+1))\right)\otimes H^0(\Omega^{r}_{n-1}(r+1))\,.
\]
\end{Lemma}
\begin{proof}
For $q<\ell$ it is trivial since $H^0(\Omega_1^{j_i}(2j_i+1-r))=0$ for all $i\in[\ell]$. For the cohomology $q=\ell$ to be non vanishing we need that $H^0(\Omega^{j_{\ell+1}}_{n-1}(2j_{\ell+1}+1-r))\neq 0$. This happens if and only if $-r+1+2j_{\ell+1}> j_{\ell+1}$. This means $j_{\ell+1}> r-1=\ell-1$, thus $j_{\ell+1}=\ell$.

Otherwise, we could have the case $H^0(\Omega^{0}_{n-1}(-r+1))\neq 0$ if and only if $r=1$. Since $r\ge 2$ such case does not hold.
\end{proof}

In the following proofs we use also the Koszul complex (we refer to \cite[Chapter III, Proposition 7.10A]{Hart} for more details)
\begin{equation}\label{eq: Koszul}
0\to \bigwedge^{\dim(\PP)+1}\sE^\ast\xrightarrow{\varphi_{\dim(\PP)}}\bigwedge^{\dim(\PP)}\sE^\ast\xrightarrow{\varphi_{\dim(\PP)-1}}\cdots\xrightarrow{\varphi_{2}}\bigwedge^2\sE^\ast\xrightarrow{\varphi_1}\sE^\ast\to \sI_{Z_T}\to 0
\end{equation}
and for all $r\ge 1$ we define the quotient bundle $\sF_i\coloneqq\left(\bigwedge^{i}\sE^\ast\right)/\mathrm{Im}(\varphi_{i})$.
After tensoring with $\sO({\bf 1})$ the complex \eqref{eq: Koszul} we get the short exact sequences
\begin{equation}\label{eq: shortexactsequence}
\begin{gathered}
0\to \sF_{r+1}({\bf 1})\to \sE^{(r)}\to \sF_r({\bf 1})\to 0\\
0\to \sF_{2}({\bf 1})\to \sE^{(1)}\to \sI_{Z_T}({\bf 1})\to 0\,.
\end{gathered}
\end{equation}
Our goal is to use the long exact sequences in cohomology of the two previous short exact sequences to compute the dimension $h^0(\sI_{Z_T}({\bf 1}))$, that is, the codimension of $\langle Z_T\rangle$ in $\PP(V)$.

\noindent Lemmas \ref{lemma: coh1}, \ref{lemma: coh2} and \ref{lemma: coh3} directly imply the next corollary. 

\begin{Corollary}\label{coh: chains}
Consider a space $V$ of order $\ell+1$ tensors of format $\nn=(2,\ldots,2,n)$ with $n\ge\ell+2$.
The following chains of isomorphisms and inclusions hold:
\begin{equation*}
\begin{gathered}
H^0(\sF_2({\bf 1}))\cong\cdots\cong H^{\ell-1}(\sF_{\ell+1}({\bf 1}))\\
H^{\ell+1}(\sF_{\ell+3}({\bf 1}))\cong\cdots\cong H^{\ell+n-2}(\sF_{\ell+n}({\bf 1}))=0\\
H^1(\sF_2({\bf 1}))\cong\cdots\cong H^{\ell-1}(\sF_\ell({\bf 1}))\subset H^{\ell}(\sF_{\ell+1}({\bf 1}))\\
H^{\ell+2}(\sF_{\ell+2}({\bf 1}))\subset\cdots\subset H^{\ell+n-1}(\sF_{\ell+n}({\bf 1}))=0\,.
\end{gathered}
\end{equation*}
\end{Corollary}

\begin{Proposition}\label{prop: format 224}
Let $T$ be a generic tensor of format $(2,2,4)$.
Then $\langle Z_T\rangle$ has dimension six in $\PP(V)\cong\PP^{15}$ and codimension one in $\PP(H_T)$.
This is the last concise format $(2,2,n)$.
\end{Proposition}
\begin{proof}
Following the similar cohomology computation in \cite[Lemma 3.5]{DOT}, we have that the vanishing of the cohomologies $H^q(\sE^{(r)})$, $q=r-1,r-2$, does not hold anymore. Furthermore, this means that computing $H^r(\sE^{(r)})$ is useful in many cases. In this case one computes that the only non-zero dimensions $h^q(\sE^{(r)})$ are
\[
h^2(\sE^{(3)})=1\,,\quad h^3(\sE^{(3)})=1\,.
\]
Consider the first short exact sequence in \eqref{eq: shortexactsequence}. The corresponding long exact sequence in cohomology is
\begin{equation}\label{longcoh}
\begin{gathered}
\cdots\to H^{r-2}(\sE^{(r)})\to H^{r-2}(\sF_r({\bf 1}))\to H^{r-1}(\sF_{r+1}({\bf 1}))\to H^{r-1}(\sE^{(r)})\to\\
\to H^{r-1}(\sF_r({\bf 1}))\to H^{r}(\sF_{r+1}({\bf 1}))\to H^{r}(\sE^{(r)})\to\cdots
\end{gathered}
\end{equation}
The sequence \eqref{longcoh} yields the following inclusions and isomorphisms:
\begin{itemize}
    \item $H^{r-2}(\sF_r({\bf 1}))\cong H^{r-1}(\sF_{r+1}({\bf 1}))$ and $H^{r-1}(\sF_r({\bf 1}))\cong H^{r}(\sF_{r+1}({\bf 1}))$ for $r\neq 3$
    \item $H^{1}(\sF_{3}({\bf 1}))\subset H^{2}(\sF_{4}({\bf 1}))$ and $H^2(\sF_3({\bf 1}))\cong H^2(\sE^{(3)})$ for $r=3$.
\end{itemize}
In turn, we get that
\begin{itemize}
    \item $H^0(\sF_2({\bf 1}))\cong H^1(\sF_3({\bf 1}))\subset H^2(\sF_4({\bf 1}))\cong H^3(\sF_5({\bf 1}))\cong H^4(\sF_6({\bf 1}))=0$
    \item $H^1(\sF_2({\bf 1}))\cong H^2(\sF_3({\bf 1}))$ and $H^3(\sF_4({\bf 1}))\cong H^4(\sF_5({\bf 1}))\cong H^5(\sF_6({\bf 1}))=0$
\end{itemize}
Therefore, if we take the second short exact sequence in \eqref{eq: shortexactsequence} and we compute the corresponding long exact sequence in cohomology, we get that
$$
0=H^0(\sF_2({\bf 1}))\to H^0(\sE^{(1)})\to H^0(\sI_{Z_T}({\bf 1}))\to H^1(\sF_2({\bf 1}))\to H^1(\sE^{(1)})=0\,,
$$
thus $h^0(\sI_{Z_T}({\bf 1}))=h^0(\sE^{(1)})+h^1(\sF_2({\bf 1}))=8+1$. This means that $\langle Z_T\rangle$ has codimension $9$ in $\PP(V)\cong\PP^{15}$, that is $\dim \langle Z_T\rangle=15-9=6$.\qedhere
\end{proof}

\begin{Theorem}\label{thm: 2...2k+2}
Let $T$ be a generic order $\ell+1$ tensor of format $(2,\ldots,2,\ell+2)$.
Then the projective span of singular $(\ell+1)$-tuples has dimension
\begin{equation}\label{eq: dim span k+1 order 2...2k+2}
\dim(\langle Z_T\rangle)=2^\ell(\ell+2)-(\ell+1)-\binom{\ell+2}{2}-\max\{0,(\ell-1)^\ell-(\ell-2)^\ell(\ell+2)\}\,.
\end{equation}
In particular $\langle Z_T\rangle=\PP(H_T)$ for $\ell\ge 4$.
\end{Theorem}
\begin{proof}
We start noticing that, by Lemmas \ref{lemma: coh1}, \ref{lemma: coh2}, \ref{lemma: coh3} and Corollary \ref{coh: chains}, only $H^\ell(\sE^{(\ell)})$ and $H^\ell(\sE^{(\ell+1)})$ are non-zero.
Furthermore, it holds
\[
\begin{gathered}
H^0(\sF_2({\bf 1}))\cong\cdots\cong H^{\ell-1}(\sF_{\ell+1}({\bf 1}))\subset H^\ell(\sF_{\ell+2}({\bf 1}))\cong\cdots\cong H^{2\ell}(\sF_{2\ell+2}({\bf 1}))=0\\
H^1(\sF_2({\bf 1}))\cong\cdots\cong H^{\ell-1}(\sF_\ell({\bf 1}))\subset H^\ell(\sF_{\ell+1}({\bf 1}))\\
H^{\ell+1}(\sF_{\ell+2}({\bf 1}))\subset\cdots\subset H^{2\ell+1}(\sF_{2\ell+2}({\bf 1}))=0\,.
\end{gathered}
\]
In simple terms, to determine the dimension of $\langle Z_T\rangle$ we need to compute $h^{\ell-1}(\sF_{\ell}({\bf 1}))$.
We first consider the long exact sequence in cohomology coming from equation \eqref{eq: shortexactsequence}. For $r=\ell+1$ we have:
\begin{equation*}
\begin{aligned}
0\to H^{\ell}(\sE^{(\ell+1)})\to H^{\ell}(\sF_{\ell+1}({\bf 1}))\to H^{\ell+1}(\sF_{\ell+2}({\bf 1}))=0\,.
\end{aligned}
\end{equation*}
Thus $h^{\ell}(\sE^{(\ell+1)})=h^{\ell}(\sF_{\ell+1}({\bf 1}))$.
For $r=\ell$ we get:
\[
0\to H^{\ell-1}(\sF_\ell({\bf 1}))\xrightarrow{\alpha}H^\ell(\sF_{\ell+1}({\bf 1}))\xrightarrow{\beta}H^{\ell}(\sE^{(\ell)})\to\cdots
\]
Notice that since $\alpha$ is injective, it is enough to determine the rank of $\beta$ to determine $h^{\ell-1}(\sF_\ell({\bf 1}))$. We recall the commutative diagram
\[
\begin{tikzcd}
H^\ell(\sE^{(\ell+1)}) \arrow[rr, "\gamma"] \arrow[rd, "\cong"] & & H^\ell(\sE^{(\ell)})\\
& H^\ell(\sF_{\ell+1}(1)) \arrow[ru, "\beta"] &
\end{tikzcd}
\]
We have that the rank of $\gamma$ is equal to the rank of $\beta$.

The associated weights to  $H^\ell(\sE^{(\ell)})$ and $H^\ell(\sE^{(\ell+1)})$ are respectively:
\begin{enumerate}
    \item $(1-\ell)\lambda_1^{(1)}\otimes\cdots\otimes (1-\ell)\lambda_1^{(\ell)}\otimes \lambda_{\ell+1}^{(\ell+1)}$.
    \item$-\ell\lambda_1^{(1)}\otimes\cdots\otimes -\ell\lambda_1^{(\ell)}\otimes \lambda_{\ell+2}^{(\ell+1)}$.
\end{enumerate}
Thus using Theorem \ref{thm: Bott}, we have that
\begin{enumerate}
    \item $H^\ell(\sE^{(\ell)})\cong G_{\ell-3}\otimes\cdots\otimes G_{\ell-3}\otimes G_{\ell+1}\cong \left(S^{\ell-3}\C^2\right)^{\otimes \ell}\otimes \bigwedge^{\ell+1}\C^{\ell+2}$.
    \item$H^\ell(\sE^{(\ell+1)})\cong G_{\ell-2}\otimes\cdots\otimes G_{\ell-2}\otimes G_{\ell+2}\cong \left(S^{\ell-2}\C^2\right)^{\otimes \ell}\otimes \bigwedge^{\ell+2}\C^{\ell+2}$.
\end{enumerate}
The map
\[
\gamma\colon(S^{\ell-2}\C^2)^{\otimes \ell}\to (S^{\ell-3}\C^2)^{\otimes \ell}\otimes \bigwedge^{\ell+1}\C^{\ell+2}
\]
acts as a contraction:
\[
\gamma(f_1\otimes\cdots\otimes f_\ell)=\sum_{i_1,\cdots,i_\ell=1}^2\partial_{i_1}f_1\otimes\cdots\otimes \partial_{i_\ell}f_\ell\otimes T_{i_1\cdots i_\ell},
\]
where
\[
T=\sum_{i_1,\cdots,i_\ell=1}^2e_{1,i_1}\otimes\cdots\otimes e_{\ell,i_\ell}\otimes T_{i_1\cdots i_\ell}\,,\quad T_{i_1\cdots i_\ell}\coloneqq \sum_{j=1}^{\ell+2}t_{i_1\cdots i_\ell,j}e_{\ell+1,j}\in \C^{\ell+2}\,.
\]
Each element $f_j\in S^{\ell-2}\C^2$ can be written as $f(x_{i,1},x_{i,2})=\sum_{d_j=0}^{\ell-2}c_{j,d_j}x_{j,1}^{\ell-2-d_j}\cdot x_{j,2}^{d_j}$. Hence a basis of $(S^{\ell-2}\C^2)^{\otimes \ell}$ is $\{\bigotimes_{j=1}^\ell x_{j,1}^{\ell-2-d_j}x_{j,2}^{d_j}\mid 0\le d_j\le \ell-2\}$. In particular
\[
\gamma\left(\bigotimes_{j=1}^\ell x_{j,1}^{\ell-2-d_j}x_{j,2}^{d_j}\right)=\sum_{i_1,\cdots,i_\ell=1}^2\bigotimes_{j=1}^\ell\partial_{i_j}(x_{j,1}^{\ell-2-d_j}x_{j,2}^{d_j})\otimes T_{i_1\cdots i_\ell}\,.
\]
If $T$ is generic, then the rank of $\gamma$ is maximal, and coincides with the minimum between the dimensions of the domain and the codomain of $\gamma$. More precisely, we conclude that
\[
\mathrm{rank}(\beta)=\mathrm{rank}(\gamma)=\min\{(\ell-1)^\ell,(\ell-2)^\ell(\ell+2)\}=
\begin{cases}
(\ell-1)^\ell & \text{for $\ell\ge 4$}\\
(\ell-2)^\ell(\ell+2) & \text{for $\ell=3$}\,.
\end{cases}
\]
This implies that
\[
h^{\ell-1}(\sF_\ell({\bf 1}))=\dim(\ker(\beta))=h^\ell(\sF_{\ell+1}({\bf 1}))-\mathrm{rank}(\beta)=(\ell-1)^\ell-\mathrm{rank}(\gamma)=0
\]
for all $\ell\ge 4$, in turn $h^{1}(\sF_2({\bf 1}))=0$. Therefore $h^0(\sI_{Z_T}({\bf 1}))=h^0(\sE^{(1)})=\ell+\binom{\ell+2}{2}$, hence
\[
\dim(\langle Z_T\rangle)=\dim(\PP(H_T))=2^\ell(\ell+2)-(\ell+1)-\binom{\ell+2}{2}\,,
\]
which agrees with \eqref{eq: dim span k+1 order 2...2k+2}.
\end{proof}

With similar techniques, we are able to prove the following result.

\begin{Theorem}\label{thm: format 23n}
Let $T$ be a generic order three tensor of format $(2,3,n)$.
\begin{itemize}
    \item[$(i)$] If $n=5$, then $\langle Z_T\rangle$ has either dimension $13$ or $14$ in $\PP(V)\cong\PP^{29}$. The expected dimension is $13$, hence there are $2$ more linear relations among the singular triples of $T$.
    \item[$(ii)$] If $n=6$, $\langle Z_T\rangle$ has either dimension $13$ or $14$ in $\PP(V)\cong\PP^{35}$. The expected dimension is $13$, hence there are $3$ more linear relations among the singular triples of $T$. This is the last concise order three format $(2,3,n)$.
\end{itemize}
\end{Theorem}

\section{Equations of the span of singular tuples in special formats}\label{sec: relations}

In Section \ref{sec: cohomology} we computed the dimension of $\langle Z_T\rangle$ for a generic tensor $T$ with the aid of cohomology tools. This section is more oriented towards the computation of the equations of $\langle Z_T\rangle$. At the current status of research, in this direction the range of possible formats covered is smaller than the more general formats studied in Section \ref{sec: cohomology}, and the results mostly rely on symbolic computations with Macaulay2 \cite{GS}. However, as explained soon, for the formats studied we can confirm our general conjecture that the tensor $T$ always belongs to the span $\langle Z_T\rangle$.

\subsection{Equations of \texorpdfstring{$\langle Z_T\rangle$}{ZT} in the format \texorpdfstring{$\nn=(2,2,n)$}{22n}}\label{sec: 22n equations}

In this format, there is only one interesting case that is for $n=4$. The format is non-concise for $n\ge 5$.

\medskip
Consider a generic tensor $U=(u_{ijk})$ of format $\nn=(2,2,4)$, and consider the set $Z_U\subset\PP(\C^2\otimes\C^2\otimes\C^4)$. From Theorem \ref{thm: FO formula}, we have that $\dim(Z_U)=0$ and $|Z_U|=8$ for a generic $U$. The 8 singular triples of $U$ may be computed numerically from the code presented in Section \ref{sec: julia}.

On one hand, the projectivized critical space $\PP(H_U)\subset\PP(\C^2\otimes\C^2\otimes\C^4)\cong\PP^{15}$ has dimension 7. Indeed, the linear relations coming from the contractions of $U$ are $\binom{2}{2}+\binom{2}{2}+\binom{4}{2}=8$ and are pairwise linearly independent by Proposition \ref{prop: dim H_T}.
On the other hand, the projective span $\langle Z_U\rangle$ is strictly contained in $\PP(H_U)$: indeed, we showed in Proposition \ref{prop: format 224} that $\dim(\langle Z_U\rangle)=6$.
Therefore, there exists an additional linear relation among the singular triples of $U$.
In Section \ref{sec: julia} we explain how to double-check this numerically by tensorizing the singular tuples previously computed.

The additional linear relation may be obtained in this way.
Let $(\xx_1,\xx_2,\xx_3)$ be a singular triple of $U$. By definition the two vectors $U(\xx_1\otimes\xx_2)$ and $\xx_3$ are proportional. From this fact we build the $4\times 4$ matrix
\[
A \coloneqq
\begin{bmatrix}
U(\xx_1\otimes\xx_2) & \xx_3 & U_{(1,1)} & U_{(1,2)}
\end{bmatrix}^T\,,
\]
where $U_{(i,j)}=(u_{ij1},\ldots,u_{ij6})$ for all $(i,j)\in[2]\times[2]$.
If $U$ is generic, we have that $\mathrm{rank}(A)=3$.
Now let $\xx_1'=(x_{1,2},x_{1,1})$ and consider the matrix
\[
A' \coloneqq
\begin{bmatrix}
U(\xx_1'\otimes\xx_2) & \xx_3 & U_{(2,1)} & U_{(2,2)}
\end{bmatrix}^T\,.
\]
In this case the first two rows of $A'$ are not proportional. We checked symbolically that still $\mathrm{rank}(A')=3$, hence the determinant of $A'$, which is linear in the coordinates $z_{ijk}$ of $\PP(\C^2\otimes\C^2\otimes\C^4)$, is contained in the ideal of $\langle Z_U\rangle$. We verified also that $\det(A')$ is linearly independent from the equations of $H_T$. Hence $\det(A')$ can be considered as ``the'' unknown additional relation among the singular triples of $U$.

Developing $\det(A')$ using the Laplace expansion along the first two rows of $A'$ and taking into account the relations $z_{ijk}=x_{1,i}x_{2,j}x_{3,k}$, we get that (we omit the computation) 
\[
\det(A')=
\begin{vmatrix}
{z}_{211}&{z}_{212}&{z}_{213}&{z}_{214}\\
{u}_{111}&{u}_{112}&{u}_{113}&{u}_{114}\\
{u}_{211}&{u}_{212}&{u}_{213}&{u}_{214}\\
{u}_{221}&{u}_{222}&{u}_{223}&{u}_{224}
\end{vmatrix}
+
\begin{vmatrix}
{z}_{221}&{z}_{222}&{z}_{223}&{z}_{224}\\
{u}_{121}&{u}_{122}&{u}_{123}&{u}_{124}\\
{u}_{211}&{u}_{212}&{u}_{213}&{u}_{214}\\
{u}_{221}&{u}_{222}&{u}_{223}&{u}_{224}
\end{vmatrix}\,.
\]
From this expression, we immediately observe that this additional relation is satisfied by the tensor $U$ itself, meaning that $[U]\in\langle Z_U\rangle$. Note the change of indices with respect to
\[
\det(A)=
\begin{vmatrix}
{z}_{211}&{z}_{212}&{z}_{213}&{z}_{214}\\
{u}_{211}&{u}_{212}&{u}_{213}&{u}_{214}\\
{u}_{111}&{u}_{112}&{u}_{113}&{u}_{114}\\
{u}_{121}&{u}_{122}&{u}_{123}&{u}_{124}
\end{vmatrix}
+
\begin{vmatrix}
{z}_{221}&{z}_{222}&{z}_{223}&{z}_{224}\\
{u}_{221}&{u}_{222}&{u}_{223}&{u}_{224}\\
{u}_{111}&{u}_{112}&{u}_{113}&{u}_{114}\\
{u}_{121}&{u}_{122}&{u}_{123}&{u}_{124}
\end{vmatrix}\,.
\]
Both determinants may be seen as bihomogeneous polynomials in the variables $u_{ijk}$ and $z_{ijk}$ of bidegree $(3,1)$. What is more, observe that in the construction of $A$ we have made a choice for the last two rows. In general, there are $6=\binom{4}{2}$ possibilities to complete the matrix $A$ using the vectors $(u_{ijk})_k$. Consider also the vector $\xx_2'=(x_{2,2},x_{2,1})$ and build the $9\times 4$ matrix
\begin{equation}\label{eq: 9x4 matrix}
\begin{bmatrix}
U(\xx_1'\otimes\xx_2) & U(\xx_1\otimes\xx_2') & U(\xx_2\otimes\xx_1) & U(\xx_2'\otimes\xx_1') & \xx_3 & U_{(1,1)} & U_{(1,2)} & U_{(2,1)} & U_{(2,2)}
\end{bmatrix}^T\,.
\end{equation}
We computed symbolically all maximal minors of the previous matrix. There are exactly $6$ of them which belong to the ideal of $\langle Z_U\rangle$. One of them is exactly the determinant of $A'$ studied above. The other five are obtained considering all remaining choices of pairs of rows $(U_{(i_1,j_1)},U_{(i_2,j_2)})$ among the last four rows, the row of $\xx_3$ and one of the first four rows (according to symmetries of the pair $(U_{(i_1,j_1)},U_{(i_2,j_2)})$ chosen).

\subsection{Equations of \texorpdfstring{$\langle Z_T\rangle$}{ZT} in the format \texorpdfstring{$\nn=(2,3,n)$}{23n}}\label{sec: 23n equations}

In this case, there are two interesting formats between the sub-boundary format and the non-concise format, precisely for $n\in\{5,6\}$.

\begin{Example}\label{ex: 2x3x5}
Consider a generic tensor $U=(u_{ijk})$ of format $\nn=(2,3,5)$. It admits $18$ singular triples, and by Proposition \ref{prop: dim H_T} the projectivized critical space $\PP(H_U)\subset\PP(\C^2\otimes\C^3\otimes\C^5)\cong\PP^{29}$ has dimension $15$.
Let $Z_U\subset\PP(\C^2\otimes\C^3\otimes\C^5)$. By Theorem \ref{thm: FO formula}, we have that $|Z_U|=18$ for a generic $U$. The projective span $\langle Z_U\rangle$ is strictly contained in $\PP(H_U)$: indeed we showed in Theorem \ref{thm: format 23n}$(i)$ that $13\le\dim(\langle Z_U\rangle)\le 14$. We verified symbolically that there exist two new relations among the singular triples, thus proving that $\dim(\langle Z_U\rangle)=13$. We write them as determinants of $5\times 5$ matrices:
\begin{align*}
\det(A_1) &=
\begin{vmatrix}
U(\xx_1'\otimes\xx_2) & \xx_3 & U_{(1,1)} & U_{(1,2)} & U_{(1,3)}
\end{vmatrix}
=
\begin{vmatrix}
T_{(1,1)}\\
U_{(2,1)}\\
U_{(1,1)}\\
U_{(1,2)}\\
U_{(1,3)}
\end{vmatrix}
+
\begin{vmatrix}
T_{(1,2)}\\
U_{(2,2)}\\
U_{(1,1)}\\
U_{(1,2)}\\
U_{(1,3)}
\end{vmatrix}
+
\begin{vmatrix}
T_{(1,3)}\\
U_{(2,3)}\\
U_{(1,1)}\\
U_{(1,2)}\\
U_{(1,3)}
\end{vmatrix}\\
\det(A_2) &=
\begin{vmatrix}
U(\xx_1'\otimes\xx_2) & \xx_3 & U_{(2,1)} & U_{(2,2)} & U_{(2,3)}
\end{vmatrix}
=
\begin{vmatrix}
T_{(2,1)}\\
U_{(1,1)}\\
U_{(2,1)}\\
U_{(2,2)}\\
U_{(2,3)}
\end{vmatrix}
+
\begin{vmatrix}
T_{(2,2)}\\
U_{(1,2)}\\
U_{(2,1)}\\
U_{(2,2)}\\
U_{(2,3)}
\end{vmatrix}
+
\begin{vmatrix}
T_{(2,3)}\\
U_{(1,3)}\\
U_{(2,1)}\\
U_{(2,2)}\\
U_{(2,3)}
\end{vmatrix}\,.
\end{align*}
Also in this case we have chosen specific vectors $(u_{ijk})_k$ to form the matrices $A_1$ and $A_2$, but there are of course other choices and all possibilities can be obtained by computing all maximal minors of a large matrix similar to the one in \eqref{eq: 9x4 matrix}.
\end{Example}

\begin{Example}\label{ex: 2x3x6}
Consider a generic tensor $U=(u_{ijk})$ of format $\nn=(2,3,6)$.
It admits $18$ singular triples, and by Proposition \ref{prop: dim H_T} the projectivized critical space $\PP(H_U)\subset\PP(\C^2\otimes\C^3\otimes\C^6)\cong\PP^{35}$ has dimension $16$.
Let $Z_U\subset\PP(\C^2\otimes\C^3\otimes\C^6)$. By Theorem \ref{thm: FO formula}, we have that $|Z_U|=18$ for a generic $U$. Also in this case the projective span $\langle Z_U\rangle$ is strictly contained in $\PP(H_U)$. By Theorem \ref{thm: format 23n}$(ii)$ we have that $13\le\dim(\langle Z_U\rangle)\le 14$, hence there are at least two and at most three new relations among singular triples.
We computed symbolically the new three linear relations in this way. Consider $\xx_1'=(x_{1,2},x_{1,1})$ and the $6\times 6$ matrices
\begin{align*}
A_1 &=
\begin{bmatrix}
U(\xx_1'\otimes\xx_2) & \xx_3 & U_{(0,0)} & U_{(0,1)} & U_{(0,2)} & U_{(1,0)}
\end{bmatrix}^T\\
A_2 &=
\begin{bmatrix}
U(\xx_1'\otimes\xx_2) & \xx_3 & U_{(0,0)} & U_{(0,1)} & U_{(0,2)} & U_{(1,1)}
\end{bmatrix}^T\\
A_3 &=
\begin{bmatrix}
U(\xx_1'\otimes\xx_2) & \xx_3 & U_{(0,0)} & U_{(0,1)} & U_{(0,2)} & U_{(1,2)}
\end{bmatrix}^T
\end{align*}
where $U_{(i,j)}=(u_{ij1},\ldots,u_{ij6})$ for all $(i,j)\in[2]\times[3]$.
Each determinant $\det(A_j)$, after the substitutions $z_{ijk}=x_{1,i}x_{2,j}x_{3,k}$, gives a linear relation among the 18 singular triples of the generic tensor $U$.
Each linear relation can be seen as a sum of $2$ determinants of $6\times 6$ matrices.
\end{Example}

The next proposition generalizes the observations made in Section \ref{sec: 22n equations} and in Examples \ref{ex: 2x3x5} and \ref{ex: 2x3x6}, and provides a method to check easily that the new relations among singular $k$-tuples of a tensor $U$ are satisfied by $U$ itself.

\begin{Proposition}\label{prop: det A}
Consider a tensor $U=(u_{i_1\cdots i_k})$ of format $\nn=(n_1,\ldots,n_k)$, where $n_k\ge 1+\sum_{i=1}^{k-1}(n_i-1)$. Consider the $n_k\times n_k$ matrix
\[
A=
\begin{bmatrix}
U(\yy_1\otimes\cdots\otimes\yy_{k-1}) & \yy_k & U_{I_1} & \cdots & U_{I_{n_k-2}}
\end{bmatrix}^T
\]
where $I_l\in\prod_{i=1}^{k-1}[n_i]$ and $U_{I_l}=(u_{j_1\cdots\,j_{k-1}j_{k}}\mid(j_1,\ldots,j_{k-1})\in I_l)$ for all $l\in[n_k-2]$, while using \eqref{eq: def contraction},
\[
U(\yy_1\otimes\cdots\otimes\yy_{k-1})_{s}=\sum_{j_\ell\in[n_\ell]}u_{j_1\cdots\,j_{k-1}s}\,y_{1,j_1}\cdots y_{k-1,j_{k-1}}\quad\forall\,s\in[n_k]\,.
\]
Then $\det(A)$ contains only terms in $y_{1,j_1}\cdots y_{k,j_{k}}$ with $(j_1,\ldots,j_{k-1})\in \prod_{i=1}^{k-1}[n_i]\setminus\{I_1,\ldots,I_{n_k-2}\}$.
\end{Proposition}
\begin{proof}
We compute $\det(A)$ by applying the generalized Laplace formula with respect to the first two rows of $A$. We use the shorthand $U_{I_l}^{(p,q)}$ to denote the row vector obtained after removing the columns $p$ and $q$ from $U_{I_l}$. We also denote by $\sigma_{p,q}$ the permutation of $[n_k]$ sending $1$ to $p$ and $2$ to $q$.
\begin{align}\label{eq: expansion det A}
\begin{split}
\det(A) &= \sum_{1\le p<q\le n_k}\mathrm{sign}(\sigma_{p,q})
\begin{vmatrix}
U(\yy_1\otimes\cdots\otimes\yy_{k-1})_{p} & U(\yy_1\otimes\cdots\otimes\yy_{k-1})_{q}\\
y_{k,p} & y_{k,q}
\end{vmatrix}
\cdot
\begin{vmatrix}
U_{{I_1}}^{(p,q)}\\
\vdots\\
U_{I_{n_k-2}}^{(p,q)}
\end{vmatrix}\\
&= \sum_{1\le p<q\le n_k}\mathrm{sign}(\sigma_{p,q})
\sum_{j_\ell\in[n_\ell]}(u_{j_1\cdots\,j_{k-1}p}z_{j_1\cdots\,j_{k-1}q}-u_{j_1\cdots\,j_{k-1}q}z_{j_1\cdots\,j_{k-1}p})
\begin{vmatrix}
U_{{I_1}}^{(p,q)}\\
\vdots\\
U_{I_{n_k-2}}^{(p,q)}
\end{vmatrix}\\
&= \sum_{j_\ell\in[n_\ell]}\sum_{1\le p<q\le n_k}\mathrm{sign}(\sigma_{p,q})
\begin{vmatrix}
u_{j_1\cdots\,j_{k-1}p} & u_{j_1\cdots\,j_{k-1}q}\\
z_{j_1\cdots\,j_{k-1}p} & z_{j_1\cdots\,j_{k-1}q}
\end{vmatrix}
\cdot
\begin{vmatrix}
U_{{I_1}}^{(p,q)}\\
\vdots\\
U_{I_{n_k-2}}^{(p,q)}
\end{vmatrix}\\
&= \sum_{j_\ell\in[n_\ell]}\det(\tilde{A}(j_1,\ldots,j_{k-1}))\,,
\end{split}
\end{align}
where in the second equality in \eqref{eq: expansion det A} we plugged in the relations $u_{j_1\cdots\,j_{k}}=y_{1,j_1}\cdots y_{k,j_{k}}$ and
\[
\tilde{A}(j_1,\ldots,j_{k-1})\coloneqq
\begin{bmatrix}
U_{(j_1,\ldots,j_{k-1})} & z_{(j_1,\ldots,j_{k-1})} & U_{{I_1}} & \cdots & U_{I_{n_k-2}}
\end{bmatrix}^T\,.
\]
Hence $\det(\tilde{A}(j_1,\ldots,j_{k-1}))\neq 0$ only if $(j_1,\ldots,j_{k-1})\in\prod_{i=1}^{k-1}[n_i]\setminus\{I_1,\ldots,I_{n_k-2}\}$, giving the desired result.
\end{proof}

Equation \eqref{eq: expansion det A} tells us that $\det(A)$ may be written as a sum of determinants of the matrices $\tilde{A}(j_1,\ldots,j_{k-1})$. The number of non-zero summands is equal to the cardinality of $\prod_{i=1}^{k-1}[n_i]\setminus\{I_1,\ldots,I_{n_k-2}\}$, that is $n_1\cdots n_{k-1}-n_k+2$. For example, we have seen in Section \ref{sec: 22n equations} that the unknown relations among singular triples of a $2\times 2\times 4$ tensor can be written as the sum of $2\cdot 2-4+2=2$ determinants. Or in Example \ref{ex: 2x3x5} that the unknown relations among singular triples of a $2\times 3\times 5$ tensor can be written as the sum of $2\cdot 3-5+2=3$ determinants. 

\medskip

\begin{Theorem}\label{thm: membership}
Let $T\in V$ be a generic tensor of order-$k$ of the following formats:
\begin{enumerate}
    \item $k=3$, $\nn=(2,2,n)$, $n\ge4$; 
    \item $k=3$, $\nn=(2,3,n)$, $n\ge5$;
    \item $k=\ell+1$, $\nn=(2,\dots,2,\ell+2)$, $\ell\ge 4$.
\end{enumerate}
Then $T\in \langle Z_T\rangle$.
\end{Theorem}
\begin{proof}
The first two items are done in the Sections \ref{sec: 22n equations} and \ref{sec: 23n equations}. The last item comes from Theorem \ref{thm: 2...2k+2}, since in such case $T\in H_T=\langle Z_T\rangle$.
\end{proof}

This result and some numerical experiments in M2 give indications that the following conjecture is true.

\begin{Conjecture}\label{con: Tspan}
Suppose $T\in V$ is a generic tensor. Then $T\in \langle Z_T\rangle$. 
\end{Conjecture}

The next result extends Theorem \ref{thm: fiber} by the second author for some special formats.

\begin{Theorem}\label{thm: fiberbf}
Let $T\in V$ be a generic tensor of order $k$ and assume Conjecture \ref{con: Tspan} holds, i.e., $T\in \langle Z_T\rangle$. Then the fiber of the rational map $\tau\colon T\mapsto \mathrm{Eig}(T)$ is $T$ itself. 
\end{Theorem}
\begin{proof}
Let $T\in \bigotimes_{i=1}^{k-1}\C^{n_i}\otimes L\subset V$ be a generic tensor of boundary format $\nn$, that is $\dim(L)=1+\sum_{i=1}^{k-1}(n_i-1)$. By Theorem \ref{thm: specialization} we have that $\langle Z_T\rangle \subset\bigotimes_{i=1}^{k-1}\C^{n_i}\otimes L$. Furthermore Theorem \ref{thm: fiber} says that the fiber of the map $\tau\colon T\mapsto\mathrm{Eig}(T)$ introduced in \eqref{eq: def tau} is one point for tensors in spaces satisfying the boundary format. Suppose that $U\in V$ is not contained in any subspace satisfying boundary format. Assuming that Conjecture \ref{con: Tspan} holds, it follows that $U\in\langle Z_U\rangle$ and $\langle Z_U\rangle$ is not contained in any subspace of boundary format. Thus $Z_U\neq Z_T$ and the fiber of the map at $\mathrm{Eig}(T)$ is a single point.

We now proceed by using the fact that the rank of the map $\tau$ satisfies semi-continuity, therefore the map is generically finite-to-one. Furthermore, since the fibers are linear spaces, we obtain that the generic fiber is a single point.
\end{proof}

\begin{Corollary}\label{cor: fiber}
If $T$ is a generic tensor of one of the formats in Theorem \ref{thm: membership}, then the fiber of the map $\tau$ is $T$ itself.
\end{Corollary}

\section{Computing singular tuples using HomotopyContinuation.jl}\label{sec: julia}

In this section we describe a code that computes numerically the singular tuples of a tensor $U$ of format $\nn$, if there are finitely many. The code uses the Julia package HomotopyContinuation.jl \cite{HomotopyContinuation.jl}. Since in all our examples we input a generic tensor $U$, then all its singular tuples $(\xx_1,\ldots,\xx_k)$ of $U$ are such that $x_{i,1}\neq 0$ for all $i\in[k]$. For this reason, in the code we will consider the following square system in the $n_1+\cdots+n_k+k$ variables $\xx_1,\ldots,\xx_k,\lambda_1,\ldots,\lambda_k$, which is obtained from \eqref{eq: def singular vector tuple matrix}:
\begin{equation}\label{eq: system svt lambdas}
\begin{cases}
U(\xx_1\otimes\cdots\otimes \xx_{i-1}\otimes \xx_{i+1}\otimes\cdots\otimes \xx_k) = \lambda_i\,\xx_i\quad\forall\,i\in[k]\\
x_{i,1}=1\quad\forall\,i\in[k]\,.
\end{cases}
\end{equation}
The value $\lambda_i$ is called the {\em $i$-th singular value} of the singular $k$-tuple $(\xx_1,\ldots,\xx_k)$. If the tensor $U$ is sufficiently generic, then all its singular tuples are non-isotropic, hence up to rescaling they are normalized. The immediate consequence is that, for every (normalized) singular tuple $(\xx_1,\ldots,\xx_k)$ of $U$ we have $\lambda_1=\cdots=\lambda_k=\lambda$. The value $\lambda$ is called the {\em singular value} associated to $(\xx_1,\ldots,\xx_k)$.

First, applying Theorem~\ref{thm: FO formula} we determine the number of singular tuples $\mathrm{ed}(\nn)$ of a generic tensor of format $\nn$ with the following function {\blue \verb|number_singular_tuples|}:

\begin{Verbatim}[commandchars=\\\{\}]
\textcolor{magenta}{using} HomotopyContinuation;
\textcolor{magenta}{function} \textcolor{blue}{number_singular_tuples}(dims)
    ldims = length(dims);
    \textcolor{purple}{@var} t[1:ldims];
    f = prod([sum([(sum(t)-t[i])^(dims[i]-1-j)*t[i]^j \textcolor{magenta}{for} j\textcolor{magenta}{=}0:(dims[i]-1)])
              \textcolor{magenta}{for} i\textcolor{magenta}{=}1:ldims]);
    (expf,cf) = exponents_coefficients(f,t);
    dims2 = map(1:ldims) \textcolor{magenta}{do} i dims[i]-1 \textcolor{magenta}{end};
    ind = findall(i -> expf[:,i]==vcat(dims2...), collect(1:size(expf,2)))[1];
    convert(Int64, cf[ind])
\textcolor{magenta}{end}
\end{Verbatim}

The function {\blue \verb|singular_tuples|} computes numerically the singular tuples of a tensor $U$:

{\small
\begin{Verbatim}[commandchars=\\\{\}]
\textcolor{magenta}{function} \textcolor{blue}{singular_tuples}(U)
    \textcolor{orange}{# (0) Preliminary settings}
    dims = size(U);
    ldims = length(dims);
    CI = CartesianIndices(U);
    \textcolor{orange}{# (1) Define the variables}
    varu = map(CI) \textcolor{magenta}{do} i Variable(\textcolor{PineGreen}{:u}, collect(Tuple(i))...) \textcolor{magenta}{end};
    varu_vector = vec(varu);
    varl = map(1:ldims) \textcolor{magenta}{do} i Variable(\textcolor{PineGreen}{:l},i) \textcolor{magenta}{end};
    varx = map(1:ldims) \textcolor{magenta}{do} i map(1:dims[i]) \textcolor{magenta}{do} j Variable(\textcolor{PineGreen}{:x}, i, j) \textcolor{magenta}{end} \textcolor{magenta}{end};
    varx_vector = vcat(varx...);
    var_vector = vcat(varx_vector,varl);
    \textcolor{orange}{# (2) Define the tensor U}
    tU = sum([varu[i]*prod([varx[j][i[j]] \textcolor{magenta}{for} j\textcolor{magenta}{=}1:ldims]) \textcolor{magenta}{for} i \textcolor{magenta}{in} CI]);
    \textcolor{orange}{# (3) Write the equations defining singular tuples}
    eq1 = [differentiate(tU,varx[i][j])-varl[i]*varx[i][j] \textcolor{magenta}{for} i\textcolor{magenta}{=}1:ldims \textcolor{magenta}{for} j\textcolor{magenta}{=}1:dims[i]];
    eq = vcat(eq1,[varx[i][1]-1 \textcolor{magenta}{for} i\textcolor{magenta}{=}1:ldims]);
    sys_tuples = System(eq; variables = var_vector, parameters = varu_vector);
    \textcolor{orange}{# (4) Write one start solution of the previous system}
    randl = rand(ComplexF64);
    L0 = vcat(map(i -> [1;zeros(dims[i]-1)], 1:ldims)...);
    L0 = vcat(L0,[randl \textcolor{magenta}{for} i\textcolor{magenta}{=}1:ldims]);
    \textcolor{orange}{# (5) Write start parameters for the start solution}
    U0 = Array\{ComplexF64\}(\textcolor{PineGreen}{undef}, dims);
    \textcolor{magenta}{for} i \textcolor{magenta}{in} CI
        \textcolor{magenta}{if} length(findall(>(1), Tuple(i))) == 0
            U0[i] = randl
        \textcolor{magenta}{elseif} length(findall(>(1), Tuple(i))) == 1
            U0[i] = 0
        \textcolor{magenta}{else}
            U0[i] = rand(ComplexF64)
        \textcolor{magenta}{end}
    \textcolor{magenta}{end};
    U0_vec = vec(U0);
    U_vec = vec(U);
    \textcolor{orange}{# (6) Track the start solution L0 to a solution for the tensor U}
    sol_1 = solve(sys_tuples, L0; start_parameters = U0_vec, target_parameters = U_vec);
    L = solutions(sol_1);
    \textcolor{orange}{# (7) Compute the number of singular tuples of U}
    ed = number_singular_tuples(dims);
    \textcolor{orange}{# (8) Find all other solutions of the system for U using monodromy}
    sol = monodromy_solve(sys_tuples, L, U_vec, target_solutions_count = ed);
    solutions(sol)
\textcolor{magenta}{end}
\end{Verbatim}
}

The following list outlines the main steps of the function {\blue \verb|singular_tuples|}:
\begin{enumerate}
    \item Introduce the variables $u_{j_1\cdots,j_k}$, $\xx_i=(x_{i,1},\ldots,x_{i,n_i})$ and $\lambda_1,\ldots,\lambda_k$.
    \item Write the $k$-linear form \verb|tU| associated to the tensor $U$. This is useful for writing \eqref{eq: system svt lambdas}.
    \item Define the system \eqref{eq: system svt lambdas} whose solutions are the singular tuples $(\xx_1,\ldots,\xx_k)$ together with their singular values $(\lambda_1,\ldots,\lambda_k)$. 
    \item Declare the start solution $(\xx_1^*,\ldots,\xx^*_k,\lambda_1^*,\ldots,\lambda_k^*)$, where $\xx_i^*=(1,0,\ldots,0)\in\C^{n_i}$ for all $i\in[k]$, $\lambda_1^*=\cdots=\lambda_k^*=\lambda^*$ and $\lambda^*$ is a randomly chosen complex number.
    \item Determine a tensor $U^*=(u_{j_1,\ldots,j_k}^*)$ (\verb|U0| in the code) such that $(\xx_1^*,\ldots,\xx^*_k,\lambda_1^*,\ldots,\lambda_k^*)$ is a solution of \eqref{eq: system svt lambdas} with respect to $U^*$. We build such a tensor by setting
    \[
    u_{j_1\cdots j_k}^*=
    \begin{cases}
    \lambda^* & \text{if $j_1=\cdots=j_k=1$}\\
    0 & \text{if $j_\ell=1$ for all $\ell\neq j$ and $j_i>1$ for some $i\in[k]$}\,,
    \end{cases}
    \]
    while $u_{j_1\cdots j_k}^*$ is a randomly chosen complex number otherwise. It is easy to verify from \eqref{eq: system svt lambdas} that $(\xx_1^*,\ldots,\xx^*_k)$ is a singular tuple of $U^*$ with singular value $\lambda^*$. 
    \item Determine one singular tuple of $U$ by tracking the start solution $(\xx_1^*,\ldots,\xx^*_k,\lambda_1^*,\ldots,\lambda_k^*)$ of \eqref{eq: system svt lambdas} with respect to the parameters $U^*$ to a solution with respect to the parameters $U$, using the method of homotopy continuation.
    \item Determine an upper bound for the number of singular tuples of $U$, namely the ED degree of the Segre product $\PP$ with respect to the format $\nn$. When $U$ is sufficiently generic, this upped bound is attained.
    \item Find all the other solutions of \eqref{eq: system svt lambdas} using monodromy.
\end{enumerate}

The output of {\blue \verb|singular_tuples|} is a list of solutions of the form $(\xx_1,\ldots,\xx_k,\lambda_1,\ldots,\lambda_k)$. The function {\blue \verb|to_tensor|} extracts the components $\xx_1,\ldots,\xx_k$ and returns the vectorization of the rank-one tensor $\xx_1\otimes\cdots\otimes\xx_k$:
{\small
\begin{Verbatim}[commandchars=\\\{\}]
\textcolor{magenta}{function} \textcolor{blue}{to_tensor}(solutions,dims)
    T = map(1:length(dims)) \textcolor{magenta}{do} i
        k = sum(dims[1:i])
        solutions[(k-dims[i])+1:k]
    \textcolor{magenta}{end}
    kron(T...)
\textcolor{magenta}{end}
\end{Verbatim}
}

Now that all the necessary functions are defined, we are ready to compute the dimension $\dim(\langle Z_U\rangle)$ with the following lines, where $U$ is randomly chosen.

{\small
\begin{Verbatim}[commandchars=\\\{\}]
nn = (2,3,4,5);
U = rand(ComplexF64,nn);
listsol = singular_tuples(U);
solVectors = map(s -> to_tensor(s,nn), listsol);
M = hcat(solVectors...);
\textcolor{magenta}{using} LinearAlgebra;
rk = rank(M);
A, S, B = svd(M, full = true);
[nn; rk; S[rk]; S[rk+1]]
\end{Verbatim}
}

The value \verb|rk| equals $\dim(\langle Z_U\rangle)+1$ with probability one. The SVD of the matrix \verb|M| computed in the last two lines ensures that the value $\verb|rk|$ is correct, in particular that the value of \verb|S[rk+1]| is very close to zero compared to \verb|S[rk]|.

Using the previous code, we determined $\dim(\langle Z_U\rangle)$ for several formats $\nn$. In Table \ref{table: dim span k=3} we display some values of $\dim(\langle Z_U\rangle)$ when $k=3$. We set $n_B\coloneqq n_1+n_2-2$ to be the value of $n_3$ such that $(n_1,n_2,n_B)$ is a boundary format. Furthermore, we denote by $\delta$ the difference between $n_B$ and the value of $n_3$ for which $\dim(\langle Z_T\rangle)$ stabilizes. A value of $\dim(\langle Z_T\rangle)$ is highlighted in red if for $n_3=n_B+\delta$ we have that $\langle Z_T\rangle$ is a proper subspace of $\PP(H_T)$. As the format size increases, the difference $\delta$ increases as well.

Instead in Table \ref{table: dim span k=l+1} we study the dimension $\dim(\langle Z_T\rangle)$ when $T$ is an order-$(\ell+1)$ tensor of format $\nn=(2,\ldots,2,n_{\ell+1})$. Again the blue values correspond to boundary formats. Similarly to Table \ref{table: dim span k=3}, a value of $\dim(\langle Z_T\rangle)$ is highlighted in red if, when $\dim(\langle Z_T\rangle)$ stabilizes, it is also strictly contained in $\PP(H_T)$ (including the blue entry for $(\ell,n_{\ell+1})=(2,3)$). For all entries not in red (except for $(\ell,n_{\ell+1})=(2,3)$), we have that $\langle Z_T\rangle=\PP(H_T)$. In particular, this confirms the statement of Theorem \ref{thm: 2...2k+2} for $n_{\ell+1}=\ell+2$.

\begin{table}[htbp]
\centering
\begin{tabular}{|c||c|c|c|c|}
	\hline
	$\nn$ & $n_B$ & \text{$\dim(\langle Z_T\rangle)$ if $n=n_B+\delta$} & $\delta$ & $\mathrm{ed}(\nn)$ \\
	\hhline{|=||=|=|=|=|}
	$(2,2,n)$ & $3$ & 6 & 0 & 8 \\
	\hline
	$(2,3,n)$ & $4$ & 13 & 0 & 18 \\
	\hline
	$(2,4,n)$ & $5$ & 22 & 0 & 32 \\
	\hline
	$(2,5,n)$ & $6$ & 33 & 0 & 50 \\
	\hline
	$(2,6,n)$ & $7$ & 46 & 0 & 72 \\
	\hhline{|=||=|=|=|=|}
	$(3,3,n)$ & $5$ & \red{29} & 1 & 61 \\
	\hline
	$(3,4,n)$ & $6$ & \red{50} & 1 & 148 \\
	\hline
	$(3,5,n)$ & $7$ & \red{76} & 1 & 295 \\
	\hhline{|=||=|=|=|=|}
	$(4,4,n)$ & $7$ & 87 & 1 & 480 \\
	\hline
	$(4,5,n)$ & $8$ & \red{133} & 2 & 1220 \\
	\hline
	$(4,6,n)$ & $9$ & \red{188} & 3 & 2624\\
	\hline
	$(4,7,n)$ & $10$ & \red{252} & 3 & 5012 \\
	\hhline{|=||=|=|=|=|}
	$(5,5,n)$ & $9$ & \red{204} & 3 & 3881 \\
	\hline
	$(5,6,n)$ & $10$ & \red{289} & 4 & 10166 \\
	\hline
	$(5,7,n)$ & $11$ & 388 & 4 & 23051 \\
	\hhline{|=||=|=|=|=|}
	$(6,6,n)$ & $11$ & \red{410} & 5 & 31976 \\
	\hline
	$(6,7,n)$ & $12$ & \red{551} & 6 & 85526 \\
	\hline
	$(6,8,n)$ & $13$ & \red{712} & 7 & 201536 \\
	\hline
\end{tabular}
\vspace*{1mm}
\caption{{\normalsize Values of $\dim(\langle Z_T\rangle)$ in the format $\nn=(n_1,n_2,n_3)$. }}\label{table: dim span k=3}
\end{table}

\begin{table}[htbp]
\centering
\begin{tabular}{|c||c|c|c|c|c|c|c|c|c|c|c|c||c|}
\hline
\backslashbox{$\ell$}{$n_{\ell+1}$} & 3 & 4 & 5 & 6 & 7 & 8 & 9 & 10 & 11 & 12 & 13 & $\cdots$ & $\mathrm{ed}(\nn)$ \\
\hhline{|=||============|=|}
2 & \blue{$6$} & $\cdots$ & & & & & & & & & & & 8 \\
3 & & \blue{$22$} & ${\red 23}$ & $\cdots$ & & & & & & & & & 48 \\
4 & & & \blue{$65$} & $76$ & $\cdots$ & & & & & & & & 384 \\
5 & & & & \blue{$171$} & $197$ & $222$ & ${\red 237}$ & $\cdots$ & & & & & 3840 \\
6 & & & & & \blue{$420$} & $477$ & $533$ & $588$ & $642$ & $695$ & ${\red 722}$ & $\cdots$ & 46080\\
\hline
\end{tabular}
\vspace*{1mm}
\caption{Values of $\dim(\langle Z_T\rangle)$ in the $(\ell+1)$-dimensional format $\nn=(2,\ldots,2,n_{\ell+1})$.}
\label{table: dim span k=l+1}
\end{table}

\section*{Acknowledgements}
We thank Giorgio Ottaviani for encouraging this project and for valuable guidance.
{A special acknowledgement goes to Paul Breiding and Sascha Timme for their support in coding the algorithm of Section \ref{sec: julia}.}
The first author is partially supported by the Academy of Finland Grant 323416.
The second author is supported by European Union's Horizon 2020 research and innovation programme under the Marie Sk\l odowska-Curie Actions, grant agreement 813211 (POEMA).

\bibliographystyle{alpha}
\bibliography{bibliography}

\end{document}